\pdfoutput=1
\documentclass[a4paper,10pt]{amsart}
 \usepackage{
  amsmath,   amscd,    amssymb,   mathrsfs,
  mathtools, fullpage, enumerate, thmtools,
  fancyhdr,  stmaryrd, tikz-cd,   etoolbox,
  chngcntr,  yfonts}
\usepackage[shortlabels]{enumitem}
\usepackage[utf8]{inputenc}

 \setlist{leftmargin=2\parindent}

\usepackage{comment}

\usepackage[pagebackref, hypertexnames=false, bookmarks=false]{hyperref}

\usepackage[capitalize]{cleveref}

\usepackage[l2tabu, orthodox]{nag}

\newenvironment{smatrix}{\left( \begin{smallmatrix} } {\end{smallmatrix} \right) }

\newcommand{\stbt}[4]{\begin{smatrix}#1 & #2 \\ #3 & #4\end{smatrix}}

\theoremstyle{plain}
\newtheorem{theorem}{Theorem}[subsection]
\newtheorem{introtheorem}{Theorem}

\newtheorem{lemma}[theorem]{Lemma}
\newtheorem{proposition}[theorem]{Proposition}
\newtheorem{corollary}[theorem]{Corollary}

\newtheorem{definition}[theorem]{Definition}
\newtheorem{notation}[theorem]{Notation}
\newtheorem{conjecture}[theorem]{Conjecture}
\newtheorem{assumption}[theorem]{Assumption}

\theoremstyle{remark}
\declaretheorem[name=Remark,sibling=theorem,qed={\lower-0.3ex\hbox{$\diamond$}}]{remark}
\declaretheorem[name=Note,sibling=theorem,qed={\lower-0.3ex\hbox{$\diamond$}}]{note}

\DeclareMathOperator{\AF}{AF}

\DeclareMathOperator{\ad}{ad}

\DeclareMathOperator{\GL}{GL}
\DeclareMathOperator{\SL}{SL}

\DeclareMathOperator{\GSp}{GSp}

\DeclareMathOperator{\Gal}{Gal}
\DeclareMathOperator{\Gr}{Gr}

\DeclareMathOperator{\Hom}{Hom}

\DeclareMathOperator{\Frob}{Frob}

\DeclareMathOperator{\Iw}{Iw}

\DeclareMathOperator{\PR}{PR}

\DeclareMathOperator{\Sym}{Sym}

\DeclareMathOperator{\gr}{gr}
\DeclareMathOperator{\imp}{imp}
\DeclareMathOperator{\loc}{loc}

\DeclareMathOperator{\BC}{BC}

\DeclareMathOperator{\pr}{pr}

\DeclareMathOperator{\tr}{tr}

\DeclareMathOperator{\As}{As}

\newcommand{\f}{\mathrm{f}}

\newcommand{\QQbar}{\overline{\QQ}}
\newcommand{\QQ}{\mathbf{Q}}
\newcommand{\RR}{\mathbf{R}}
\newcommand{\FF}{\mathbf{F}}
\newcommand{\Ql}{\QQ_\ell}
\newcommand{\Qp}{\QQ_p}
\newcommand{\fN}{\mathfrak{N}}
\newcommand{\ZZ}{\mathbf{Z}}
\newcommand{\Zp}{\ZZ_p}

\newcommand{\cF}{\mathcal{F}}

\newcommand{\cK}{\mathcal{K}}
\newcommand{\cL}{\mathcal{L}}

\newcommand{\cO}{\mathcal{O}}
\newcommand{\cP}{\mathcal{P}}
\newcommand{\cR}{\mathcal{R}}
\newcommand{\cS}{\mathcal{S}}

\newcommand{\sF}{\mathscr{F}}

\newcommand{\cyc}{\mathrm{cy}}

\numberwithin{equation}{section}

\renewcommand{\ge}{\geqslant}
\renewcommand{\geq}{\geqslant}

\newcommand{\AFh}{\widehat{\AF}}
\newcommand{\cAF}{{}_c\!\AF}
\newcommand{\cAFh}{{}_c \AFh} 

\newcommand{\gAFh}{{}_\gamma \AFh}

\newcommand{\myqed}{\pushQED{\qed} \qedhere \popQED}

\author{David Loeffler}

\author{Sarah Livia Zerbes}

\title{An Euler system for the adjoint of a modular form}

\begin{document}
 \renewcommand{\crefrangeconjunction}{--} 

\begin{abstract}
 We construct an Euler system for the adjoint Galois representation of a modular form, using motivic cohomology classes arising from Hilbert modular surfaces. We use this Euler system to give an upper bound for the Selmer group of the adjoint representation over the cyclotomic $\Zp$-extension, which agrees with the predictions of the Iwasawa main conjecture up to powers of $p$.
\end{abstract}

 \maketitle

 \makeatletter
  \patchcmd{\@tocline}
  {\hfil}
  {\leaders\hbox{\,.\,}\hfil}
  {}{}
 \makeatother

 \setcounter{tocdepth}{1}
 \tableofcontents

\section{Introduction}

 \subsection{Setting}

  Many conjectures -- such as the Birch--Swinnerton-Dyer conjecture, the Bloch--Kato conjecture, and the Main Conjecture of Iwasawa theory -- relate the special values of $L$-functions attached to algebraic varietes, or more generally to Galois representations, to the algebraic properties of those Galois representations. The method of \emph{Euler systems}, originating in the work of Thaine, Kolyvagin, and Rubin, is among the most powerful tools known for attacking these conjectures, via giving upper bounds for Selmer groups.

  One particularly important case for these conjectures is the 3-dimensional \emph{adjoint} Galois representation attached to an elliptic curve over $\QQ$ (or more generally a modular form). As well as being a natural and interesting example of a Galois representation in its own right, this case is important because of its relation to modularity lifting, and thus to the proof of Fermat's Last Theorem. The Bloch--Kato conjecture is known in this setting (by \cite{diamondflachguo04}, building on the methods introduced by Taylor and Wiles \cite{taylorwiles95} for the proof of Fermat's Last Theorem); but the Iwasawa main conjecture remains open.

  In an earlier work \cite{loefflerzerbes19}, we studied the arithmetic of the adjoint using the \emph{Beilinson--Flach elements}, introduced in \cite{leiloefflerzerbes14}. These give an Euler system for a twist of the 4-dimensional tensor square Galois representation, which is the direct sum of the adjoint representation and a 1-dimensional complementary piece. However, this does not work as one might hope, due to an unexpected obstacle: the Euler factors appearing in the norm relations for these classes contain an extra term which is the Euler factor of the 1-dimensional summand. Since the Galois group acts trivially on this summand, this degree 1 factor \emph{always vanishes} at the trivial character, and these trivial zeroes render the Beilinson--Flach classes useless for bounding the adjoint Selmer group. (We were able to bound the Selmer group of the \emph{twist} of the adjoint by any sufficiently non-trivial Dirichlet character, but not that of the adjoint itself.)

 \subsection{Overview of the results}

  In this paper, we construct an Euler system for the adjoint via a different approach: as a ``degenerate'' case of the \emph{Asai--Flach elements}, introduced in \cite{leiloefflerzerbes18}. These give an Euler system for the \emph{Asai}, or \emph{twisted tensor}, Galois representation for a Hilbert modular form over a real quadratic field. We show that if the Hilbert modular form is the base-change (Doi--Nagunuma lift) of an elliptic modular form, the four-dimensional Asai representation decomposes -- like the tensor square -- into the direct sum of the adjoint representation and a 1-dimensional complement. However in contrast to the tensor square case, the Galois group acts non-trivially on the 1-dimensional complement: it acts via the quadratic character associated to the real quadratic field. This quadratic twist resolves all the problems previously experienced in bounding the Selmer group -- we may choose all the auxiliary primes in the Euler-system argument to be inert in the real quadratic field, and then the extra Euler factor is a unit in the group ring.

  Using these classes, together with a result of Byeon \cite{byeon01} showing that we may choose the real quadratic field so its $p$-adic $L$-function is a unit, we obtain a bound for the Selmer group of the adjoint over the cyclotomic $\Zp$-extension (under various technical hypotheses on $f$ and $p$, see Theorem \cref{thm:main} below for the precise statement).

  The bound we give here coincides with the bound predicted by the Iwasawa main conjecture, up to controlling powers of $p$. We hope to refine the bounds to remove the indeterminacy, and hence prove the full Main Conjecture, in future work.

 \subsection{Relation to other work}

  We note that another, very different parallel approach to constructing an Euler systems for the adjoint of a modular form has been announced by Marco Sangiovanni and Chris Skinner, in work developed in parallel with this one. Their methods are very different than ours (using the pullback to $\GL_2 \times \GL_2$ of an Eisenstein class for the larger group $\GSp_4$, rather than the pushforward of an Eisenstein class for a subgroup as in our approach).

  We also remark that a \emph{lower} bound for the Selmer group was already announced in an an unpublished preprint of Urban from 2006 (\emph{Groupes de Selmer et fonctions $L$ $p$-adiques pour les representations modulaires adjointes}), using congruences between Klingen Eisenstein series and cusp forms on $\GSp_4$.

\section{Statement of the theorem}

 \subsection{Galois representations}
  \label{sect:Galrep}

  Let $f$ be a cuspidal modular newform of weight $k+2\geq 2$, level $N_f$ and trivial nebentype\footnote{This assumption is imposed for simplicity; the general case presents no additional conceptual difficulties, but it is the trivial-character case which is of most interest for applications.}. If $L \subset \RR$ is a finite extension of $\QQ$ containing the coefficients of $f$, and $v$ a prime of $L$ above some rational prime $p$, we write
  \[ \rho_{f, v} : \Gal(\overline{\QQ} / \QQ) \to \GL_2(L_v) \]
  for an arbitrary representative of the unique isomorphism class of irreducible $L_v$-linear $\Gal(\overline{\QQ} / \QQ)$-representations which is unramified at all primes $\ell \nmid p N_f$ and satisfies $\tr \rho_{f, v}(\Frob_\ell^{-1}) = a_\ell(f)$ for all such $\ell$. (Here $\Frob_\ell$ is an arithmetic Frobenius, and thus $\Frob_\ell^{-1}$ a geometric one.) The Hodge--Tate weights of $\rho_{f, v}$ at $p$ are $0$ and $-1-k$; our convention is that the Hodge--Tate weight of the cyclotomic character is $+1$. We can, and do, suppose that $\rho_{f, v}$ takes values in $\GL_2(\cO_v)$, where $\cO_v$ is the ring of integers of $L_v$.

 \subsection{Symmetric square $L$-functions}

  We write $L(\Sym^2 f, s)$ for the complex symmetric square $L$-series of $f$, which can be written as
  \[ L(\Sym^2 f, s) = \prod_{\text{$\ell$ prime}} P_\ell(\Sym^2 f, \ell^{-s})^{-1} \]
  where $P_\ell(\Sym^2 f, X) \in 1 + X \cO_L[X]$ is characterised by
  \[ P_\ell(\Sym^2 f, X) = \det\left(1 - X \Frob_\ell^{-1} : \left(\Sym^2 \rho_{f, v}\right)^{I_\ell} \right)\]
  for any prime $v$ not above $\ell$, where $I_\ell$ is the inertia group. 
  For a Dirichlet character $\psi$ we define the twisted $L$-function $L(\Sym^2 f \times \psi, s)$ analogously; we shall only need the case when $\psi$ has conductor $N_\psi$ coprime to $N_f$, in which case
  \[ L(\Sym^2 f \times \psi, s) = \prod_{\ell \nmid N_\psi} P_\ell(\Sym^2 f, \psi(\ell) \ell^{-s})^{-1}.\]

  \begin{proposition}
   The critical values (in the sense of Deligne) for $L(\Sym^2 f \times \psi, s)$ are the integers in the range $\{1, \dots, k+1\}$ with $(-1)^s = -\psi(-1)$, and the integers in $\{k+2, \dots, 2k + 2\}$ with $(-1)^s = \psi(-1)$. \qed
  \end{proposition}

 \subsection{Canonical periods and $p$-adic $L$-functions}
  \label{sect:padicL}

  Let us choose a prime $v$ of the coefficient field $L$ (above the rational prime $p$). As in \cite[\S 2]{raysujathavatsal23}, we define a \emph{normalised period} associated to $f$ (at $v$) by
  \[ \Omega_f^{\mathrm{can}} \coloneqq \frac{\langle f, f \rangle_{\Gamma_1(N_f)}}{\eta_f} \in \RR^\times,\]
  where $\langle f, f \rangle_{\Gamma_1(N_f)} \coloneqq \int_{\Gamma_1(N_f) \backslash \mathcal{H}} |f(i + iy)|^2 y^k \mathrm{d}x\, \mathrm{d}y$ is the Petersson norm of $f$, and $\eta_f \in \cO_{L, (v)}$ is a generator of the congruence ideal of $f$ at $v$ (well-defined up to a unit in $\cO_{L, (v)}^\times$). By construction, this period has the property that for all $h \in S_k(\Gamma_1(N_f), \cO_{L, (v)})$, the ratio $\frac{1}{\Omega^{\mathrm{can}}_f}\cdot \langle f, h\rangle_{\Gamma_1(N)}$ is in $\cO_{L, (v)}$.

  We now impose the following additional hypotheses:
  \begin{itemize}
   \item $p > 2$;
   \item $p$ (or, more precisely, the prime $v \mid p$) is a good ordinary prime of $f$;
   \item the residual representation $\bar{\rho}_{f, v}$ is absolutely irreducible;
   \item $f$ is \emph{$p$-distinguished}, i.e.~the two characters occurring in the semisimplification of $\bar\rho_{f, v} |_{G_{\Qp}}$ are distinct.
  \end{itemize}
  In our setting the $p$-distinction condition can be stated more concretely as the condition that either $(p - 1) \nmid (k + 1)$, or $a_p(f)^2 \ne 1 \bmod v$; in particular, if $k = 0$ it holds for all $p > 2$.

  \begin{proposition}[Schmidt, Ray--Sujatha--Vatsal]
   There exists a measure $L_p(\Sym^2 f) \in \cO_v[[\Zp^\times]]$ with the following property: if $s \in \ZZ$, and $\chi$ is a Dirichlet character of $p$-power conductor such that $L(\Sym^2 f \times \chi^{-1}, s)$ is a critical value, then we have
   \[ L_p(\Sym^2 f)(s + \chi) = (\star) \cdot \frac{L(\Sym^2 f \times \chi^{-1}, s)}{\Omega_f^{\mathrm{can}}} \]
   (with both sides lying in $\overline{\QQ}$). Moreover, this $p$-adic $L$-function satisfies a functional equation relating its values at $\sigma$ and $2k + 3 - \sigma$ (for $\sigma$ an arbitrary character of $\Zp^\times$).
  \end{proposition}

  Here $(\star)$ denotes an explicit product of various elementary factors (Gauss sums, powers of $\pi$, etc) and Euler factors at $p$; see \cite{loefflerzerbes19} for explicit formulae. Note that the critical $L$-values $L(\Sym^2 f \times \chi^{-1}, s)$ are finite and non-zero, but if $\chi = 1$ then the Euler factors in $(\star)$ vanish for $s = k + 1$ and $s = k + 2$, so $L_p(\Sym^2 f)$ has ``trivial'' zeroes at these points.

  \begin{proof}
   The construction of the $p$-adic $L$-function as an element of the \emph{non-integral} Iwasawa algebra $\cO_v[[\Zp^\times]][1/p]$ is carried out in \cite{schmidt88} (with some small additions due to later authors to take care of the interpolation property when $s = k + 1$ or $k + 2$, see Remark 2.3.3 of \cite{loefflerzerbes19}). For a proof that it in fact lies in the integral Iwasawa algebra, see \cite{raysujathavatsal23}.

   (The results of \emph{op.cit.} are stated for a slightly different setting, applying to the twisted $L$-function $L(\Sym^2 f \otimes \psi, s)$ for any character with $\psi^2 \ne 1 \bmod v$; but in fact all that is required is that none of the characters $\psi \omega_\nu$ are trivial mod $v$, where $\omega_\nu$ varies over a certain family of auxiliary \emph{odd} quadratic characters (see Corollary 2.9 of \emph{op.cit.}). So this condition is also satisfied if $\psi = 1$, and the proof of integrality of the $p$-adic $L$-function goes through in this case also.)
  \end{proof}

  It will be convenient to work with a shifted version of this function, since we want to consider the arithmetic of the adjoint of $f$, which is a twist of $\Sym^2 f$.

  \begin{notation}
   We let $L_p(\ad f)^{\cyc}$ be the image of $\operatorname{tw}_{k + 2} L_p(\Sym^2 f)$ in the quotient $\cO_v[[\Zp^\times / \mu_{p-1}]]$, so the value of $L_p(\ad f)^{\cyc}$ at a character $\sigma$ trivial on $\mu_{p-1}$ is $L_p(\Sym^2 f, k + 2 + \sigma)$.
  \end{notation}

 \subsection{The Selmer group}

  We introduce the following notation (cf.~\cite[\S 3.2]{loefflerzerbes19}):

  \begin{itemize}
   \item $V$ denotes the 3-dimensional, $L$-linear Galois representation $\ad(\rho_{f, v}) = \left(\Sym^2 \rho_{f, v}\right)^*(-k-1)$, for a prime $v$ satisfying the assumptions of \cref{sect:padicL}.

   \item $T$ denotes a choice of $G_{\QQ}$-stable $\cO_v$-lattice in $V$.

   \item $A$ denotes the divisible torsion Galois module $\Hom_{\Zp}(T, \mu_{p^\infty}) = \ad(\rho_{f, v}) \otimes_{\Zp} \Qp/\Zp(1)$.

   \item $\QQ^{\cyc}$ denotes the cyclotomic $\Zp$-extension of $\QQ$ (the unique $\Zp$-extension contained in $\QQ(\mu_{p^\infty})$).

   \item $\Gamma^\cyc$ denotes $\Gal(\QQ^{\cyc} / \QQ)$, which is identified via the cyclotomic character with $\Zp^\times / \mu_{p-1} \cong \Zp$; and $\Lambda^\cyc$ the Iwasawa algebra $\cO_v[[\Gamma^\cyc]]$.
  \end{itemize}

  As noted in \emph{op.cit.}, we have a natural corank 2 $G_{\Qp}$-stable submodule $\sF^1 A \subset A$. We can use this to define a Greenberg Selmer group $H^1_{\Gr, 1}\left(\QQ^{\cyc}, A\right)$, as in Proposition 3.4.4 of \emph{op.cit.}. In general this may depend on $T$; but if the image of $\rho_{f, v}$ is sufficiently large that $\Sym^2 \bar{\rho}_{f, v}$ is irreducible, then this module is (up to isomorphism) independent of the choice of $T$, since any two lattices in $V$ are homothetic and therefore isomorphic as Galois modules.

  \begin{conjecture}[Greenberg's Iwasawa main conjecture for $T$]
   The Pontryagin dual $H^1_{\Gr, 1}\left(\QQ^{\cyc}, A\right)^\vee$ is a torsion module over over $\Lambda^\cyc$, and its characteristic ideal is generated by $L_p(\Sym^2 f)^\cyc$.
  \end{conjecture}

  \begin{remark}
   We have followed \cite{greenberg89} here in our formulation of the conjectures, except in one detail: Greenberg defines two variants (``strict'' and ``non-strict'') of his Selmer groups. The general Main Conjecture is formulated with the non-strict variant, while the ``Greenberg Selmer group'' used in \cite{loefflerzerbes19} and here is the strict one; but the difference between strict and non-strict Selmer groups is easily seen to be trivial our case (using the formulae of \cite[\S 9.6.2]{nekovar06} for example), since $G_{\Qp^\cyc}$ acts on $A / \sF^1 A$ via a nontrivial unramified character.
  \end{remark}

 \subsection{Big image}

  For most of our results we will require the following condition on $f$ and $v$:

  \begin{notation}
   We say $f$ has \emph{big Galois image} at $v$ if $\rho_{f, v}\left(\Gal(\QQbar / \QQ^{\mathrm{ab}})\right)$ is the whole of $\operatorname{SL}_2(\cO_v)$.
  \end{notation}

  By a well-known result of Momose and Ribet (see \cite{ribet85}), if $f$ has no ``extra twists'' this condition holds for all but finitely many $v$; more generally, if $f$ is not of CM-type, it holds for a positive-density set of $v$. (In particular, if $f$ corresponds to a non-CM elliptic curve, then $f$ has big image at $p$ for all but finitely many primes $p$, and there are effective algorithms for determining the finitely many exceptions.)

  \subsection{The main theorem} We can now state our main theorem. We first recapitulate the assumptions:
  \begin{itemize}
   \item $f$ is a cuspidal modular newform of weight $k + 2 \ge 2$ and trivial nebentype.
   \item $p > 2$ is a rational prime not dividing the level of $f$, and $v$ a prime of the coefficient field above $p$, at which $f$ is ordinary, $p$-distinguished, and has big Galois image.
   \item The extension $[L_v : \Qp]$ is unramified.
   \item If $p = 3$ then $[L_v : \Qp] > 1$.
  \end{itemize}

  \begin{introtheorem}
   \label{thm:main}
   The Pontryagin dual $H^1_{\Gr, 1}\left(\QQ^{\cyc}, A\right)^\vee$ is a torsion $\Lambda^{\cyc}$-module, and we have the divisibility of characteristic ideals
   \[ \operatorname{char}_{\Lambda^{\cyc}}\left( H^1_{\Gr, 1}\left(\QQ^{\cyc}, A\right)^\vee\right) \ \Big|\  p^t \cdot L_p(\ad f)^{\cyc} \]
   for some integer $t$.
  \end{introtheorem}

\section{An Euler system for the symmetric square}\label{s:symES}

 \subsection{Asai Galois representations}

  Let $F$ be a real quadratic field of discriminant $D$, and write $\varepsilon_F$ for the quadratic Dirichlet character corresponding to $F$.

  For a newform $f$ as in \S\ref{sect:Galrep}, let $\pi$ be the unitary cuspidal automorphic representation of $\GL_2/\QQ$ generated by $f$, and let $\Pi=\BC(\pi)$ be the base-change (Doi--Nagunuma lift) of $\pi$ to an automorphic representation of $\GL_2/F$. Write $\cF$ for the normalised newform generating $\Pi$, which also has coefficients in $L$.

  \begin{note}
   $\fN$ always divides $N_f\cO_F$, with equality if $(D,N_f) = 1$.
  \end{note}

  Let $p>2$ be a prime coprime to $N_f$, and let $v$ be a prime of $L$ dividing $p$.

  \begin{definition}
   We write
   \[ \rho^{\As}_{\Pi,v}:\, \Gal(\overline{\QQ}/\QQ)\rightarrow \GL_4(L_v)\]
   for the Asai representation attached to $\Pi$.
  \end{definition}

  \begin{note}
   As noted in §4.4 of [LLZ18, Definition 4.3.2], using the \'etale cohomology of the Hilbert modular variety $Y_{G,1}(\mathfrak{N})$, we can construct a canonical $L_v$-vector space $V^{\As}_{\Pi}$ on which Galois acts via $\rho^{\As}_{\Pi, v}$; and there is a natural invariant $\cO_v$-lattice $T^{\As}_{\Pi}$.
  \end{note}

  \begin{proposition}
   We have an isomorphism of $L_v$-linear Galois representations
   \[
    \rho^{\As}_{\Pi,v} \cong \left[ \Sym^2  \rho_{f, v} \right]
    \oplus \left[\varepsilon_F \cdot \cyc^{-(1+k)} \right],
   \]
   where ``$\cyc$'' denotes the $p$-adic cyclotomic character.
  \end{proposition}

  \begin{proof}
   This can be seen via a comparison of traces. (Cf.~Remark 7.6(ii) of \cite{loefflerwilliams18} in the imaginary quadratic case.)
  \end{proof}

  \begin{notation}
   Henceforth, $V$ denotes the uniquely-determined subrepresentation of $(V^{\As}_{\Pi})^*(-1-k)$ isomorphic to $\ad \rho_{f, v}$, and $T$ its intersection with $(T^{\As}_{\Pi})^*(-1-k)$.
  \end{notation}

  \begin{assumption}
   We assume that $p$ is split in $F$, and that $f$ is ordinary at $p$.
  \end{assumption}

  \begin{note}\label{note:symfil}
   The restriction of $V$ to $G_{\Qp}$ inherits a $3$-step decreasing filtration $\mathscr{F}^\bullet$ with
   \begin{align*}
    \gr^0T&\cong L_v(\operatorname{unr}(\alpha^2) \cdot \cyc^{-(k+1)}),\\
    \gr^1T&\cong L_v,\\
    \gr^2T&\cong L_v(\operatorname{unr}(\alpha^{-2}) \cdot \cyc^{(k+1)}),
   \end{align*}
   where $\operatorname{unr}(x)$ denotes the unramified charater sending arithmetic Frobenius to $x$, and $\cyc$ the cyclotomic character. Note in particular that the $G_{\Qp}$-action on the top and bottom graded pieces is non-trivial modulo $v$ if and only if $f$ is $p$-distinguished. We use the same notation $\sF^i$ to denote the induced filtrations on $T$ and on $A = (V / T)(1)$.
  \end{note}

 \subsection{The Euler system classes}

  \begin{notation} \
   \begin{enumerate}[(i)]
    \item For $m \ge 1$ a square-free integer coprime to $p$, let $\Delta_m$ denote the maximal $p$-group quotient of $(\ZZ / m)^\times$, and $\QQ(m)$ denote the maximal subfield of $\QQ(\mu_m)$ of $p$-power degree, so $\Gal(\QQ(m) / \QQ) \cong \Delta_m$. We let $\QQ^{\cyc}(m) = \QQ^{\cyc} \cdot \QQ(m)$ and we write $\Gamma_m^\cyc \cong \Gamma^\cyc \times \Delta_m$ for its Galois group.

    \item For $\mathcal{K}$ an abelian extension of $\QQ$ (possibly infinite) and $0 \ne a \in \ZZ$, we let $\sigma_a \in \operatorname{Gal}(\mathcal{K}/\QQ)$ be the arithmetic Frobenius at $a$ if $a$ is a prime unramified in $\mathcal{K}$; and we extend multiplicatively to all integers $a$ coprime to the conductor.
   \end{enumerate}
  \end{notation}

  \begin{definition}\label{def:AFss}
   Given any choice of $c > 1$ coprime to $6 p N_f$, the construction of \cite{leiloefflerzerbes18} gives rise to a canonical family of Iwasawa cohomology classes
   \[ \cAF_{\cyc, m}(\Pi) \in H^1_{\Iw}\left(\QQ^{\cyc}(m),T\right), \]
   for all squarefree $m \ge 1$ coprime to $p c D N_f$.
  \end{definition}

  More precisely, we defined in \emph{op.cit.}~a family of classes ${}_c \mathcal{AF}_{a, m}(\Pi) \in H^1_{\Iw}(\QQ(\mu_{mp^\infty}), (T^{\As}_{\Pi})^*)$ for all integers $m \ge 1$, depending on a choice of $a$ generating $\cO_F / (m \cO_F + \ZZ)$. The class above is obtained from this by:
  \begin{itemize}
   \item choosing $a$ to be the unique generator of $\cO_F / \ZZ$ with $\tr_{F/\QQ}(a / \sqrt{D}) = 1$;

   \item identifying the Iwasawa cohomology of $(T^{\As}_{\Pi})^*$ and $(T^{\As}_{\Pi})^*(-1-k)$ over $\QQ(\mu_{mp^\infty})$, since the cyclotomic twist becomes trivial over this field;

   \item applying the norm map from $\QQ(\mu_{mp^\infty})$ to its subfield $\QQ^{\cyc}(m)$.
  \end{itemize}
  By the same argument as \cite[Corollary 4.1.3]{loefflerzerbes19} in the Rankin--Selberg case, this class in fact takes values in the rank 3 summand $T$ of the rank 4 module $(T_{\Pi}^{\As})^*(-1-k)$.

  \begin{proposition}\label{prop:IwasawaAFclass}
   These classes have the following properties:
   \begin{enumerate}[(i)]
    \item If $c, d$ are two integers $> 1$ coprime to $6p N_f$, then for all squarefree $m$ coprime to $p cd D N_f$, we have
    \[ d^2 (1 - \sigma_d^2) \cdot \cAF_{\cyc, m}(\Pi) =
    c^2 (1 - \sigma_c^2) \cdot {}_d \mathrm{AF}_{\cyc, m}(\Pi) \]
    \item The classes $\cAF_{\cyc, m}(\Pi)$, for varying $m$, satisfy a slightly modified Euler system norm relation: if $\ell \nmid p c D m N_f$ is a prime, we have
    \[
    \mathrm{norm}^{\ell m}_m\left( \cAF_{\cyc,\ell m}(\Pi)\right )=
    - \ell^{k+2} \sigma_\ell \cdot Q_\ell (\Pi, \sigma_\ell^{-1}) \cdot \cAF_{\cyc,m}(\Pi)
    \]
    for $Q_\ell(\Pi, X) \in \cO_v[X]$ an explicit polynomial congruent mod $(\ell - 1) \cO_v[X]$ to $P_\ell(\Pi,\ell^{-(k+2)} X)$, where
    \begin{equation}\label{eq:NRBC}
     P_\ell(\Pi,\ell^{-(k+2)} X) = \left(1-\tfrac{\varepsilon_F(\ell)}{\ell} X\right) \cdot P_\ell(\ad f, \ell^{-1} X)
    \end{equation}
    is the local $L$-factor of $\left(\ad(\rho_{f, v}) \oplus \varepsilon_F\right)(1)$ at $\ell$.
    \item The image of $\cAF_{\cyc, m}(\Pi)$ under localisation at $p$  is contained in the cohomology of the codimension 1 subrepresentation $\sF^1 T$ of \cref{note:symfil}.
   \end{enumerate}
  \end{proposition}


 \subsection{Getting rid of the parasitic norm relation factor}\label{ss:parasitic}

  The aim of this section is to show that the classes may be modified to get rid of the parasitic factors $ (1-\tfrac{\varepsilon_F(\ell)}{\ell} X)$ in the norm relations, so that we obtain an Euler system for the representation $V$.

  \begin{notation}\label{def:setofprimes}
   Let $\mathcal{P}$ denote the set of primes $\ell$ such that $\ell \nmid p D N_f$, $\ell$ is inert in $F$, and $\ell = 1 \bmod p$. Let $\mathcal{R}$ be the set of square-free products $\ell_1 \dots \ell_n$ of primes in $\mathcal{P}$.
  \end{notation}

  \begin{definition}
  If $m \in \cR$ and $\ell \mid m$, we define $\hat\sigma_\ell \in \Gamma^\cyc_m$ as the unique lifting of $\sigma_\ell \in \Gamma^\cyc_{m/\ell}$ which acts trivially on $\QQ(\ell)$.
  \end{definition}

  \begin{lemma}\label{lem:parasiticunit}
   Suppose $m \in \mathcal{R}$ and $\ell \mid m$ is a prime. Then the element $1 - \ell^{-1} \varepsilon_F(\ell) \hat\sigma_\ell^{-1}$ is a unit in the Iwasawa algebra $\cO_v[[\Gamma^\cyc_m]]$.
  \end{lemma}

  \begin{proof}
   As $\Gamma^\cyc_m$ is a $p$-group, its Iwasawa algebra is a local ring whose maximal ideal is the kernel of the augmentation map to the residue field $k_v$. Since $\ell = 1 \bmod p$, and $\ell$ is inert in $F$ (so $\varepsilon_F(\ell) = -1$), this map sends $1 - \ell^{k+1} \varepsilon_F(\ell) \hat\sigma_\ell^{-1}$ to $2 \in k_v$, which is a unit (as we have assumed $p > 2$).
  \end{proof}

  \begin{proposition}
   \label{prop:modifyES}
   There exists a compatible family of classes
   \[ \cAFh_{\cyc,m}(\Pi) \in H^1_{\Iw}\left(\QQ^{\cyc}(m), T\right),\quad m \in \cR, \quad (c, m) = 1,\]
   with the following properties:
   \begin{enumerate}[(i)]
    \item If $\ell m \in \cR$ with $\ell$ prime, the Euler system norm relation
    \[ \operatorname{norm}^{\ell m}_m(\cAFh_{\cyc,\ell m}(\Pi)) = P_\ell(\ad f,  \ell^{-1} \sigma_\ell) \cdot \cAFh_{\cyc,m}(\Pi)\]
    holds.
    \item If $\chi$ is a Dirichlet character factoring through $\Delta_m$ and having conductor exactly $m$, then the images of $\cAFh_{\cyc, m}(\Pi)$ and $\cAF_{\cyc, m}(\Pi)$ under the projection map
    \[ \pr_{\chi} : H^1_{\Iw}(\QQ^{\cyc}(m), V) \to H^1_{\Iw}(\QQ^{\cyc}, V(\chi^{-1}))\]
    are related by
    \[
     \pr_{\chi}\left(\cAF_{\cyc, m}(\Pi)\right) =
     \left(\prod_{\ell \mid m}-\ell^{k+1}(1 + \ell \chi(\hat\ell))\right)\cdot
     \pr_{\chi} \left(\cAFh_{\cyc,m}(\Pi)\right),\]
    where $\hat\ell = 1 \bmod \ell$ and $\hat\ell = \ell \bmod m/\ell$. In particular, for $m = 1$ we have $\cAFh_{\cyc,1}(\Pi) = \cAF_{\cyc, 1}(\Pi)$.
   \end{enumerate}
  \end{proposition}

  \begin{proof}
   We build up $\cAFh_{\cyc,m}(\Pi)$ from $\cAF_{\cyc, m}(\Pi)$ in three steps, all of which keep the $m = 1$ class unchanged. Firstly, we set $\mathbf{x}_m = \left(\prod_{\ell \mid m} - \ell^{k+2}\hat\sigma_\ell\right)^{-1} \cAF_{\cyc, m}(\Pi)$, so that we have $\operatorname{norm}^{\ell m}_m(\mathbf{x}_{\ell m}) = Q_\ell(\cF, \sigma_\ell^{-1}) \mathbf{x}_m$. Secondly, since we have $Q_\ell(\cF, X) = P_\ell(\cF, \ell^{-(k+2)}X) \bmod {\ell - 1}$, we can apply Lemma IX.6.1 of \cite{rubin00}, which allows us to construct a new collection of classes $\mathbf{y}_m$ from the $\mathbf{x}_m$ which have the polynomials $P_\ell(\cF, \ell^{-(k+2)} X)$ in their norm relations, and such that $\pr_\chi(\mathbf{x}_m) = \pr_\chi(\mathbf{y}_m)$ for primitive characters $\chi$. Finally, using \cref{lem:parasiticunit} we can complete the proof by setting $\cAFh_{\cyc,m}(\Pi) = \left(\prod_{\ell \mid m} (1 + \ell^{-1}\hat\sigma_\ell^{-1})\right)^{-1} \mathbf{y}_m$.
  \end{proof}


 \subsection{Local factors at bad primes}

  As we shall see below, the Asai--Flach classes are related to the ``imprimitive'' Asai $L$-series $L_{\As}^{\mathrm{imp}}(\cF, s)$. This is a Dirichlet series whose coefficients are given by an explicit formula in terms of Hecke eigenvalues of $\cF$, but whose Euler factors at primes $\ell \mid N_f$ can differ by finitely many Euler factors from the ``true'' (primitive) Asai $L$-series, whose $L$-factors are defined via the local Langlands correspondence. Thus there are polynomials $C_\ell \in \cO_L[X]$, for each $\ell \mid N_f$, such that
  \[
   L_{\As}^{\mathrm{imp}}(\cF, s) = \left(\prod_\ell C_\ell(\ell^{-s})\right) \cdot L_{\As}(\cF, s).
  \]
  Our next proposition shows that, after giving away a uniform power of $p$, we may refine our Asai--Flach classes to have ``optimal'' local Euler factors:

  \begin{proposition}\label{prop:localfactors}
   We can find an integer $t \ge 0$, and a family of classes $\{\cAFh{}_{\cyc,m}^{\mathrm{pr}}(\Pi) \in H^1_{\Iw}(\QQ^\cyc, T) : m \in \cR\}$ (``primitive Asai--Flach elements'') which satisfy the Euler system norm relations, such that
   \[
    p^t \cAFh_{\cyc, 1}(\Pi) = \left(\prod_{\ell \mid N_f} C_\ell(\ell^{-(k+2)} \sigma_\ell^{-1})\right) \cdot \cAFh{}_{\cyc, 1}^{\mathrm{pr}}(\Pi).
   \]
  \end{proposition}

  \begin{proof}
   Let $S$ denote the set of primes dividing $N_f$. Exactly as in the $\operatorname{GSp}_4$ case described in \cite{LSZ17} and \cite{LZ-equivar}, for any choice of ``test data'' consisting of an $L$-valued Schwartz function $\Phi_S$ on $\QQ_S^2 = \prod_\ell \QQ_\ell^2$, and an element $w_S$ in the Whittaker model of $\Pi_S$ (defined over $L$), we can construct a family of Iwasawa cohomology classes in which the local data (the choice of Eisenstein class, and the choice of map from the cohomology of the infinite-level Shimura variety to $V$) are determined by $(w_S, \Phi_S)$. By definition, the classes $\cAF_{\cyc,m}(\Pi)$ are the output of this construction if we take $(w_S, \Phi_S)$ to be the ``imprimitive test data'' of \cite[\S 7.3.2]{grossiloefflerzerbesLfunct}, so $w_S$ is the normalised new vector, and $\Phi_S$ a suitable multiple of the indicator function of $(0, 1) + N_f \ZZ_S^2$. However, we may repeat the construction with any choice of test data; and the resulting system of classes will have denominators bounded uniformly in $m$ (depending only on the denominators of $w_S$ and $\Phi_S$ relative to some lattice), and will satisfy the same norm relations as $\cAF_{\cyc,m}(\Pi)$ as $m$ varies, so we can modify them to obtain the exact Euler system norm relation as above. This construction defines a map from test data at the primes in $S$ to Euler systems.

   If we choose a character $\lambda$ of $\Gamma^{\cyc}$, then the image of the above class under evaluation at $\lambda$ (giving an element of $H^1(\QQ, V(\lambda^{-1}))$) will satisfy an equivariance property: for each $\ell \in S$ it will transform by an unramified character under the action of $H(\Ql)$ on $w_\ell$ and $\Phi_\ell$. From the multiplicity-one results proved in \cite{loeffler-ggp}, the maximal quotient of $\cS(\QQ_S^2) \otimes \Pi_S$ on which $H$ acts via this character is 1-dimensional, and we can construct a non-zero linear functional on this one-dimensional quotient as the product of the local Asai zeta-integrals. By definition, evaluating $\prod_\ell C_\ell(\ell^{-k-2} \sigma_\ell^{-1})$ at $\lambda$ gives the ratio of the values of the Asai zeta-integral at the stanard imprimitive test data, and at an optimal test vector. So the cohomology classes arising from the standard imprimitive test data, and the optimal test data, must stand in this same ratio. Since this holds for every $\lambda$, the underlying Iwasawa cohomology classes are related by $\prod_\ell C_\ell(\ell^{-k-2} \sigma_\ell^{-1})$. So if we define $\cAFh{}_{\cyc, 1}^{\mathrm{pr}}(\Pi)$ as the Euler system arising from an optimal test vector, scaled by a power of $p$ to cancel out its denominator, we obtain the stated result.
  \end{proof}

  \begin{remark}
   We do not know if the (inexplicit) test data which give the primitive Asai Euler factors are contained in the lattice in $\cS(\QQ_S^2) \otimes \Pi_S$ which gives rise to integral cohomology classes. (See \cite{groutides25ramified} for a partial result in this direction, covering primes $\ell \in S$ that split in $F$.) So we can construct Euler system classes without denominators, or Euler system classes with optimal local factors at the bad primes; but not both at once.

   However, since the polynomials $C_\ell$ all have constant term 1, the correction term $\prod_\ell C_\ell(\ell^{-(k+2)} \sigma_\ell^{-1})$ is coprime to $p$ in the Iwasawa algebra. Since our main theorem concerns characteristic ideals, which are insensitive to ``codimension $> 1$'' information, the above result will be sufficient for our purposes.
  \end{remark}

\section{Consequences of the big-image assumption}

 \subsection{Special elements in the image}

  Let $v$ be a big image prime for $f$, and $F$ any quadratic extension in which $p$ is unramified. (In this section we do not $f$ to be ordinary or $p$-distinguished.)

  \begin{lemma}\label{cor:existtau}
   There exists an element $\tau \in \Gal(\overline{\QQ} / \QQ)$ such that
   \begin{itemize}
    \item $\tau$ acts trivially on $\QQ(\mu_{p^\infty})$,
    \item $\tau$ maps to the nontrivial element of $\Gal(F / \QQ)$,
    \item $\rho_{f, v}(\tau)$ is conjugate in $\SL_2(\cO_v)$ to $\stbt 1 1 0 1$.
   \end{itemize}
  \end{lemma}

  \begin{proof}
   It suffices to note that $F$ and $\QQ(\mu_{p^\infty})$ are linearly disjoint over $\QQ$ (since one is unramified at $p$ and the other totally ramified), so there exists a $\tau_0 \in \Gal(\overline{\QQ} / \QQ(\mu_{p^\infty}))$ with $\tau_0 |_{F}$ non-trivial. Thus $\rho_{f, v}(\tau_0) \in \SL_2(\cO_v)$; and since the image of $\Gal(\overline{\QQ} / \QQ^{\mathrm{ab}})$ is all of $\SL_2(\cO_v)$, we can find some $\tau_1 \in \Gal(\overline{\QQ} / F(\mu_{p^\infty}))$ mapping to $\rho_{f, v}(\tau_0)^{-1}\stbt1101$. Hence $\tau = \tau_0 \tau_1$ has the desired property.
  \end{proof}

  For $r \ge 1$, let $\cK_r$ denote the composite of $\QQ(\mu_{p^r})$ and the splitting field of the representation $\rho_{f, v} \bmod \varpi^r$, for $\varpi$ a uniformizer of $\cO_v$; and let $\cL_r = F \cK_r$. This is a finite extension of $\QQ$. Note that $\cL_r$ is unramified outside $p D N_f$; and any prime $\ell \nmid p c D N_f$ whose Frobenius in $\Gal(\cL_r / \QQ)$ is conjugate to the image of $\tau$ is necessarily inert in $F$ and 1 modulo $p^r$. So $\ell$ is in the set $\mathcal{P}$ of \cref{def:setofprimes}. By the Chebotarev density theorem there are infinitely many such primes. Moreover, if $T_r = T / \varpi^r T$ and $A_r = T_r^\vee(1) = A[\varpi^r]$, then both $A_r / (\Frob_\ell - 1) A_r$ and $T_r / (\Frob_\ell - 1) T_r$ are free of rank 1 over $\cO / \varpi^r$.

\subsection{Vanishing of some residual Galois cohomology groups}

  \begin{proposition}[Hertzig]
   Let $q$ be an odd prime power. If $q \ne 5$, then we have $H^1\left(\SL_2(\FF_q), \ad\right) = 0$, where $\ad$ denotes the adjoint representation on $\mathfrak{sl}_{2, \FF_q}$. If $q = 5$, then this space is 1-dimensional over $\FF_5$, and $\GL_2(\FF_q)$ acts on it by the character $\det^2$.
  \end{proposition}

  \begin{proof}
   Theorem 1 (i) of \cite{hertzig69} shows that $H^1(\SL_2(\FF_q), \ad)$ vanishes for $q \ne 5$. The case $q = 5$ can be verified by a direct calculation.
  \end{proof}

  Now suppose $v$ is a big-image prime for $f$, with residue field $\FF_q$; and let $\mathcal{K}_1 = \QQ(\bar{\rho}, \mu_p)$, as above.

  \begin{proposition}
   Let $j \in \ZZ/ (p-1)\ZZ$. If $q = 5$, suppose that $j \ne 2 \pmod 4$. Then we have $H^1\left(\mathcal{K}_1 / \QQ, \ad(\bar{\rho})(j)\right) = 0$. In the exceptional case $(q, j) = (5, 2)$ this group has dimension 1 over $\FF_5$.
  \end{proposition}

  \begin{proof}
   Let us temporarily write $\Gamma$ for $\Gal(\cK_1 / \QQ)$, and $\Delta$ for the normal subgroup $\Gal(\cK_1 / \QQ(\mu_p))$. Evidently $\bar{\rho}$ gives a map $\Gamma \to \GL_2(\FF_q)$. The restriction of this map to $\Delta$ is injective, and takes values in $\SL_2(\FF_q)$. On the other hand, since $\Gamma / \Delta$ is abelian, $\bar{\rho}(\Delta)$ must contain the derived subgroup of $\operatorname{Im}(\bar{\rho})$. For $q > 3$, the group $\SL_2(\FF_q)$ is equal to its own derived subgroup, so we have $\bar{\rho}(\Delta) = \SL_2(\FF_q)$. For $q = 3$ this is not true ($\SL_2(\FF_3)$ is solvable), but since $\rho$ is odd, our hypotheses force $\bar{\rho}(\Gamma)$ to be all of $\GL_2(\FF_3)$, which does have $\SL_2(\FF_3)$ as its derived subgroup, and hence $\bar{\rho}(\Delta) = \SL_2(\FF_q)$ in this case also.

   The inflation-restriction sequence for $M = \ad(\bar{\rho})(j)$ now reads
   \[ 0 \to H^1(\Gamma / \Delta, M^{\Delta}) \to H^1(\Gamma, M) \to H^1(\Delta, M)^{\Gamma / \Delta}.\]
   The first term is evidently 0, since $\Gamma / \Delta$ has order coprime to $p$. We have just seen that $H^1(\Delta, M) = 0$ for $q \ne 5$. In the case $q = 5$, we have seen that $H^1(\Delta, M) \cong \FF_5$, and $\Gamma / \Delta \cong \FF_5^\times$ acts via $(\det \bar{\rho})^{2} \cdot (\omega_5)^j$, where $\omega_5$ is the mod $5$ cyclotomic character. Since $\det \bar{\rho} = \omega_5^{-1-k}$ for an even integer $k$, this is equal to $\omega_5^{j - 2}$, and accordingly the invariants are zero if (and only if) $j \ne 2 \bmod 4$.
  \end{proof}

  \begin{remark}
   We will only use this lemma for $j = 0$ and $j = 1$. The $j = 1$ case (under much weaker assumptions on $\rho$) is \cite[Proposition 1.11]{wiles95}. The $j = 0$ case, assuming $\FF_q$ is the prime field $\FF_p$ and $p > 3$, is part of \cite[Lemma 1.2]{flach92}.
  \end{remark}

 \subsection{D\'evissage}

  We now lift to the $\cO_v$-linear representations. Let $\mathcal{K}_\infty = \bigcup_n \mathcal{K}_n$ as above. Our aim is to prove the following:

  \begin{proposition}\label{prop:nWnstarW}
   Suppose $L_v / \Qp$ is an unramified extension with residue field $\FF_q$. Suppose that $q \ne 3$, and if $q = 5$ that $j \ne 2 \bmod 4$. Then the group $H^1(\mathcal{K}_\infty / \QQ, \ad(\rho_{f, v}) \otimes \Qp/\Zp(j))$ is zero.
  \end{proposition}

  \begin{proof}
   This follows by the same argument as in Lemma 1.2 of \cite{flach92}.
  \end{proof}

  \begin{remark}
   Note that Flach's argument relies on the assertion that the central extension
   \[ 1 \to \ad \to \SL_2(\cO_v / p^2) \to \SL_2(\FF_q) \to 1\]
   does not split as a semidirect product (i.e.~its class in $H^2(\SL_2(\FF_q), \ad)$ is non-zero). If $q = p$ for a prime $p \ge 5$ this is a well-known theorem of Serre, see \cite[\S IV.3.4]{serre68b}. For the extension to all non-prime finite fields other than $\FF_3$ (including the characteristic 2 cases), see Proposition 3.7 of \cite{manoharmayum15}. In contrast, the extension genuinely is split for $q = 3$ (or if $L_v / \Qp$ is ramified).
  \end{remark}

  The above argument does not involve the real quadratic field $F$. However, letting $\mathcal{L}_\infty = F \mathcal{K}_{\infty}$ as before, and noting that $\Gal(\cL_\infty / \cK_\infty)$ has order 1 or 2 and hence has vanishing cohomology with coefficients in $\Zp$-modules, we have $H^1\left(\mathcal{L}_\infty / \QQ, \ad(\rho_{f, v}) \otimes \Qp/\Zp(j)\right) = 0$ likewise.

 \subsection{Removing the $c$-factor}\label{ss:choiceofc}

  From now on we assume $f$ satisfies all of the assumptions of \cref{thm:main}. Then we have the following:

  \begin{theorem}\label{thm:divX}
   There exist classes  $\AFh_{\cyc, m}(\Pi) \in H^1_{\Iw}(\QQ^{\cyc}(m), T)$, for each $m \in \cR$, which satisfy the Euler system norm relations and such that
   \[ \cAFh_{\cyc, m}(\Pi) = c^2 (1 - \sigma_c^2) \cdot \AFh_{\cyc, m}(\Pi) \]
   for all $c > 1$ coprime to $6p m N_f$.
  \end{theorem}

  \begin{proof}
   We choose a topological generator $\gamma$ of $\Gamma^{\cyc}$ and let $X = [\gamma] - 1 \in \Lambda^{\cyc}$ (so that $\Lambda^{\cyc} = \cO_v[[X]]$). From big image, we have $H^0(\QQ(m), T) = 0$ for all $m$. Hence multiplication by $X$ is injective on $H^1_{\Iw}(\QQ^\cyc(m), T)$, and the cokernel $H^1_{\Iw}(\QQ^\cyc(m), T) / X$ is identified with a subspace of $H^1(\QQ(m), T)$.

   We have $c^2 (1 - \sigma_c^2) \in X\Lambda^{\cyc}$, for any $c > 1$ coprime to $p$, and the ratio is a unit if $\sigma_c$ generates $\Gamma^{\cyc}$. It follows that if we set, for each $m$,
   \[
    \gAFh_{\cyc, m}(\Pi) \coloneqq
    \left(\frac{d^2 (1 - \sigma_d^2)}{X}\right)^{-1} {}_d \AFh_{\cyc, m}(\Pi) \in H^1_{\Iw}(\QQ^{\cyc}(m), T),
   \]
   for some choice of $d > 1$ (depending on $m$) such that $\sigma_d$ generates $\Gamma^{\cyc}$ and maps to 1 in $\Delta_m$, then $\gAFh_{\cyc, m}(\Pi)$ is independent of the choice of $d$; and for any $c$ (not necessarily mapping to 1 in $\Delta_m$) we have
   \[ X \cdot \cAFh_{\cyc, m}(\Pi) = c^2 (1 - \sigma_c^2) \cdot \gAFh_{\cyc, m}(\Pi).\]

   For any given $m$, since $\Delta_m$ is a cyclic group of (odd) $p$-power order, we may choose $c$ such that $\sigma_c^2$ is a generator; it follows that $\gAFh_{\cyc, m}(\Pi) \bmod X \in H^1(\QQ(m), T)$ maps to 0 under any non-trivial character of $\Delta_m$, i.e.~it is $\Delta_m$-invariant.

   We show in Appendix A below that these conditions in fact imply that $\gAFh_{\cyc, m}(\Pi) \bmod X = 0$ for all $m$, so $\gAFh_{\cyc, m}(\Pi)$ is divisible by $X$ in $H^1_{\Iw}(\QQ(m), T)$. We have seen that this group has no $X$-torsion, so the quotients $\AFh_{\cyc, m}(\Pi) = X^{-1} \gAFh_{\cyc, m}(\Pi)$ are uniquely-defined and satisfy the Euler system norm relations.
  \end{proof}

  \begin{remark}
   Since $H^0(\QQ(m) \otimes \Qp, \gr^0 T) = 0$, it follows that $H^1_{\Iw}(\QQ^\cyc(m) \otimes \Qp, \gr^0 T)$ has no $X$-torsion. So, as $X\cdot \AFh_{\cyc, m}$ maps to 0 in $H^1_{\Iw}(\QQ^\cyc(m) \otimes \Qp, \gr^0 T)$, the same is also true of $\AFh_{\cyc, m}$; that is, the classes $\AFh_{\cyc, m}$ satisfy the local condition at $p$ defined by the subrepresentation $\sF^1 T$. We may also argue similarly with the ``primitive'' classes $\AFh{}^{\mathrm{pr}}_{\cyc, m}$.
  \end{remark}

\section{Bounding the Selmer group}
\label{sec:Selmer1}

 \subsection{A first Selmer group bound}

  We first consider the auxiliary Selmer group $H^1_{\Gr, 2}(\QQ^{\cyc}, A)$, in which the local condition at $p$ is given by $\sF^2 A \subset \sF^1 A$; note that this local condition is the dual of the local condition satisfied by our Euler system classes. (Compare Theorem 5.4.1 of \cite{loefflerzerbes19}, which is the analogue for sufficiently non-trivial twists of $T$ rather than $T$ itself.)

  \begin{theorem}
   \label{thm:1stSelmerbound}
   Suppose that $\AFh_{\cyc, 1}(\Pi)$ is non-zero as an element of $H^1_{\Iw}(\QQ^\cyc, T)$. Then $H^1_{\Iw, \Gr, 1}(\QQ^{\cyc}, T)$ is free of rank one over $\Lambda^\cyc$, $H^1_{\Gr, 2}(\QQ^{\cyc}, A)^\vee$ is torsion, and we have the divisibility of characteristic ideals
   \[
     \operatorname{char}_{\Lambda^\cyc} \left(H^1_{\Gr, 2}(\QQ^{\cyc}, A)^\vee\right)
     \mathrel{\Big|}
     \operatorname{char}_{\Lambda^\cyc} \left(\frac{H^1_{\Iw, \Gr, 1}(\QQ^{\cyc}, T)}{\AFh_{\cyc, 1}(\Pi)}\right).
   \]
  \end{theorem}

  \begin{proof}
   This follows from Corollary 12.3.5 of \cite{KLZ17} (taking the $T^+$ of \emph{op.cit.} to be $\sF^1 T$), with the modifications explained in \cite[Appendix]{loefflerzerbes19} to deal with the fact that our auxiliary primes have to be inert in $F$. In order to apply the machinery, we have to verify the conditions $(H.0^\sharp)-(H.5^\sharp)$ of the appendix of \cite{loefflerzerbes19}. Condition $(H.2^\sharp)$ is \cref{cor:existtau}; and Condition $(H.3^\sharp)$ is \cref{prop:nWnstarW} (for $j = 0$ and $j = 1$); the remaining conditions are straightforward.
  \end{proof}

  We now apply Poitou--Tate duality to bound the Selmer group of $A$ with the $\sF^1$ local condition.

  \begin{definition}
   Denote by
   \[ \cL^{\PR} : H^1_{\Iw}(\QQ^\cyc_{p}, \gr^1 T) \to \Lambda^\cyc \otimes \gr^1 T \]
   the Perrin-Riou regulator map.
  \end{definition}

  \begin{proposition}
   Let $\xi \in \gr^1 T^*$ be a vector giving a choice of isomorphism
   \[ \gr^1 T \cong \cO_v.\]
   If $\cL^{\PR}\left(\AFh_{\cyc, 1}(\Pi)\right) \ne 0$, then we have
   \[ \operatorname{char}_{\Lambda^\cyc}
    \left(\frac{H^1_{\Iw}(\QQ_{\infty, p}^\cyc, \gr^1 T)}
     {\loc_p\left( \AFh_{\cyc, 1}(\Pi) \right)}\right) =
   \left\langle \cL^{\PR}\left(\AFh_{\cyc, 1}(\Pi)\right), \xi \right\rangle.\]
  \end{proposition}

  \begin{proof}
   By Theorem 8.2.3 and Remark 8.2.4 of \cite{KLZ17}, we hence have an exact sequence
   \[
    0\rightarrow \cO_v \rightarrow
    H^1_{\Iw}(\QQ_{\infty, p}^\cyc, \gr^1 T)\rightarrow \Lambda^{\cyc}\rightarrow  \cO_v \rightarrow 0,
   \]
   where the middle map is given by $\langle \cL^{\PR}(-),\, \xi\rangle$. Since by assumption $\loc_p\left( \AFh_{\cyc, 1}(\Pi)\right)$ does not lie in the kernel of this map, the result follows.
  \end{proof}

  By a standard Poitou--Tate duality computation (compare \cite[Theorem 11.6.4]{KLZ17}) we obtain the following:

  \begin{corollary}
   \label{thm:2ndSelbound}
   If $\cL^{\PR}\left(\AFh_{\cyc, 1}(\Pi)\right) \ne 0$, then $H^1_{\Gr, 1}(\QQ^{\cyc}, A)$ is a cotorsion $\Lambda^\cyc$-module, and we have the divisibility
   \[
    \operatorname{char}_{\Lambda^\cyc} \left(H^1_{\Gr, 1}(\QQ^{\cyc}, A)^\vee\right)\ \big|\ \left\langle\cL^{\PR}\left(\AFh_{\cyc, 1}(\Pi)\right)\!, \xi \right\rangle.
   \]
   Moreover, equality holds iff equality holds in \cref{thm:1stSelmerbound}.\myqed
  \end{corollary}

  Using the classes $\left(\AFh{}^{\mathrm{pr}}_{\cyc, m}(\Pi)\right)_{m \in \cR}$, we also have the analogous bound with $\AFh_{\cyc, 1}(\Pi)$ replaced by $\AFh{}^{\mathrm{pr}}_{\cyc, 1}(\Pi)$. In the following section we show that $\left\langle\cL^{\PR}\left(\AFh{}^{\mathrm{pr}}_{\cyc, 1}(\Pi)\right)\!, \xi \right\rangle$ is related to $L_p(\ad f)^{\cyc}$ (c.f. Theorem \ref{thm:ERL}).


 \subsection{An explicit reciprocity law}

  As above, let $\xi \in \gr^1 T^*$ be a vector giving a choice of isomorphism
  \[ \gr^1 T \cong \cO_v.\] Then the explicit reciprocity law of \cite{grossiloefflerzerbes-GO4} shows that
  \[
   \left\langle \cL^{\PR}\left(\AFh_{\cyc, 1}(\Pi)\right), \xi \right\rangle =
    R_{\xi} \cdot L_{p, \As}^{\imp}(\Pi),
  \]
  where $L_{p, \As}^{\imp}(\Pi)$ is the imprimitive Asai $p$-adic $L$-function defined in \cite{grossiloefflerzerbesLfunct}, and $R_\xi \in L^\times$ is a constant depending on $\xi$ and the choice of periods used to define $L_{p, \As}^{\imp}(\Pi)$. From \cref{prop:localfactors} it follows that we have
  \[ \left\langle \cL^{\PR}\left(\AFh{}_{\cyc, 1}^{\mathrm{pr}}(\Pi)\right), \xi \right\rangle =
      p^t R_{\xi} \cdot L_{p, \As}(\Pi), \]
  where $L_{p, \As}(\Pi)$ is the \emph{primitive} Asai $L$-function.

  \begin{theorem}\label{thm:krekovthesis}
   We have
   \[L_{p, \As}(\Pi) = R_\xi' \cdot L_p(\ad f,s)^{\cyc} \cdot L_{p}(\varepsilon, s+1),
   \]
   where $R_{\xi}' \in L^\times$ is another constant factor.
  \end{theorem}

  \begin{proof}
   This is the main result of the forthcoming ETH Z\"urich PhD thesis of Dimitrii Krekov, generalising the well-known result of Dasgupta \cite{Dasgupta-factorization} for $p$-adic Rankin--Selberg $L$-functions.
  \end{proof}

  As a consequence, we obtain the following bound for the Selmer group:

  \begin{corollary}\label{thm:ERL}
   The $\Lambda^{\cyc}$-module $H^1_{\Gr, 1}(\QQ^{\cyc}, A)$ is cotorsion, and we have the divisibility
    \[
    \operatorname{char}_{\Lambda^\cyc} \left(H^1_{\Gr, 1}(\QQ^{\cyc}, A)^\vee\right)\ \big|\ p^t \cdot
    L_p(\ad f, \sigma)^{\cyc} \cdot L_{p}(\varepsilon, \sigma + 1)
    \]
    for some $t \in \ZZ$.
  \end{corollary}

\subsection{Removing the Dirichlet character}

To dispose of the Dirichlet-character term, we recall the following beautiful theorem from \cite{byeon01}:

\begin{theorem}[Byeon]\label{thm:byeon}
	Let $p \ge 3$ be prime and let $F$ be the real quadratic field $\QQ(\sqrt{p^2 + 4})$. Then $p$ splits in $F$, and $L_p(\varepsilon_F, 1)$ is a $p$-adic unit.
\end{theorem}

\begin{remark}
	Byeon also shows that for $p > 3$ there are infinitely many more fields $F$ which have this property, but a single field suffices for our purposes.
\end{remark}

\begin{corollary}[Theorem A]\label{cor:prelim}
	Let $p \ge 3$ be prime. Assume that
	\begin{itemize}
		\item $f$ is a cuspidal modular newform of weight $k + 2 \ge 2$ and trivial nebentype.
		\item $p > 2$ is a rational prime not dividing the level of $f$, and $v$ a prime of the coefficient field above $p$, at which $f$ is ordinary, $p$-distinguished, and has big Galois image.
		\item The extension $[L_v : \Qp]$ is unramified.
		\item If $p = 3$ then $[L_v : \Qp] > 1$.
	\end{itemize}
	Then the Pontryagin dual $H^1_{\Gr, 1}\left(\QQ^{\cyc}, A\right)^\vee$ is a torsion $\Lambda^{\cyc}$-module, and we have the divisibility of characteristic ideals
	\begin{equation}\label{eq:Xbound}
		\operatorname{char}_{\cO_v[[\Gamma^\cyc]]}\left( H^1_{\Gr, 1}\left(\QQ^{\cyc}, A\right)^\vee\right) \ \Big|\  p^t\cdot L_p(\ad f)^{\cyc},
	\end{equation}
	for some $t \ge 0$.
\end{corollary}

\begin{proof}
 Since $L_p(\varepsilon_F, 1)$ is a $p$-adic unit, the $p$-adic $L$-function $L_p(\varepsilon_F, \sigma + 1)$ is a unit in the Iwasawa algebra of the cyclotomic $\Zp$-extension. The result therefore follows from \cref{thm:ERL}.
\end{proof}

\appendix

\section{Residues of Euler systems}
 \label{sect:app-residues}

 \subsection{Setup}

  We suppose $L/\Qp$ is finite with ring of integers $\cO$; $T$ is a finite-rank free $\cO$-module with continuous $G_{\QQ}$-action, unramified outside a finite set $\Sigma \ni p$; and $\cP$ is a set of primes such that $(T, \cP$ satisfies the conditions (H.0)--(H.5) of \emph{op.cit.}. We shall also suppose $T$ has rank $> 1$, so (H.1) implies that $H^0(\QQ, T \otimes k_L) = H^0(\QQ, T^*(1) \otimes k_L) = 0$.

  Let $\cR$ be the set of squarefree products of primes in $\cP$, and $\mathbf{z} = (z_m)_{m \in \cR}$ a set of elements in $H^1_{\Iw}(\QQ^{\cyc}(m), T)$ satisfying the Euler system norm relation\footnote{For concreteness we use the normalisations of \cite{rubin00}, rather than \cite{mazurrubin04}, the difference being a factor of $\ell$ in the norm relation.}. We write $c_m$ for the image of $z_m$ in $H^1(\QQ(m), T)$, and we suppose the following condition holds:

  \begin{assumption}
   For every $m \in \cR$ and every \emph{non-trivial} character $\chi$ of $\Delta_m$, the image of $c_m$ in $H^1(\QQ, T(\chi))$ is zero.
  \end{assumption}

  We want to deduce from this that in fact $c_m = 0$ for all $m$, so that $z_m$ is divisible by $\gamma - 1$ in $H^1_{\Iw}(\QQ^\cyc(m), T)$ for all $m$. From the norm relations, it suffices to prove that $c_1 = 0$.

 \subsection{Kolyvagin classes}

  For $M \in \cO$, let $\bar{c}_1$ be the image of $c_1 \in H^1(\QQ, T/MT)$. It suffices to show that $\bar{c}_1 = 0$ (for an arbitrary $M$), since $H^1(\QQ, T)$ has no divisible elements.

  We define derived Kolyvagin classes $\kappa_{[F, m, M]}$ as in \cite{rubin00}; we shall only need the case when $m = \ell$ is a prime in $\cP$, and $F = \QQ$, so we write simply $\kappa_{[\ell, M]}$. Here $\ell$ is assumed to satisfy the conditions
  \begin{itemize}
   \item $M \mid \ell - 1$,
   \item $M \mid P_\ell(1)$.
  \end{itemize}

  Under our present assumptions, $\kappa_{[\ell, M]}$ can be characterised as the unique class in $H^1(\QQ, T / MT)$ whose image in $H^1(\QQ(\ell), T / M T)$ is equal to $D_\ell c_\ell \bmod M$, where $D_\ell = \sum_{i = 1}^{|\Delta_\ell| - 1} i \sigma^i$, for a choice of generator $\sigma$ of the cyclic group $\Delta_\ell$.

  \begin{proposition}
   For $\ell$ as above, the class $\kappa_{[\ell, M]}$ is zero.
  \end{proposition}

  \begin{proof}
   It suffices to show that $D_\ell c_\ell = 0 \bmod M$ in $H^1(\QQ(\ell), T)$.

   By hypothesis, $c_\ell$ has zero image under all non-trivial characters of $\Delta_\ell$, so $c_\ell$ is $\Delta_\ell$-invariant as an element of $H^1(\QQ(\ell), T[1/p])$. Since $H^0(\QQ, T \otimes k_L) = 0$, and $\QQ(\ell) / \QQ$ is a $p$-extension, we have $H^0(\QQ(\ell), T \otimes k_L) = 0$ as well; it follows that $H^1(\QQ(\ell), T)$ is torsion-free and thus $c_\ell$ is $\Delta_\ell$-invariant. Thus $D_\ell$ acts on $c_\ell$ as multiplication by $\sum_{i = 1}^{|\Delta_\ell|-1} i = 0 \bmod M$.
  \end{proof}

  \begin{corollary}
   For any prime $\ell$ as above, we have $\phi_\ell^{\mathrm{fs}}\left(\loc_{\ell} \bar{c}_1\right) = 0$, where $\phi_\ell^{\mathrm{fs}} : H^1_{\mathrm{f}}(\Ql, T/MT) \to \frac{H^1(\Ql, T/MT)}{H^1_{\f}(\Ql, T/MT)}$ is the ``finite-singular comparison'' map \cite[Definition 4.5.3]{rubin00}.
  \end{corollary}

  \begin{proof}
   This follows from the previous proposition and Theorem 4.5.4 of \cite{rubin00}.
  \end{proof}

  \begin{remark}
   Note that the proof of this result fundamentally uses the existence of the Iwasawa cohomology classes $z_m$ mapping to $c_m$; it would not be enough just to have the $c_m$ alone.
  \end{remark}

  In particular, if $\ell$ is chosen such that $(T / MT) / (\Frob_\ell - 1)$ is free of rank 1 over $\cO / M\cO$, then $\phi_\ell^{\mathrm{fs}}$ is a bijection and we can conclude that $\loc_{\ell} c_1 = 0 \bmod M$. However, Lemma 5.2.1 of \cite{rubin00} which shows that among primes satisfying these conditions, there is a positive-density subset such that $\loc_{\ell} (c_1 \bmod M)$ has the same order as $c_1 \bmod M$. So we obtain a contradiction unless $c_1 = 0 \bmod M$. Since $M$ was arbitrary, the result follows.

  \begin{remark}
   More generally, given a pair of representations $(T_1, T_2)$ such that $(T_1, T_2, \cP)$ satisfies the conditions (H.$0^\sharp$)--(H.$5^\sharp$) of \cite[Appendix]{loefflerzerbes19}, and an Euler system for $T_1$ satisfying the above conditions, we again obtain the same conclusion using the slight modifications to Rubin's arguments explained \emph{loc.cit.}.
  \end{remark}

\providecommand{\noopsort}[1]{\relax}
\providecommand{\bysame}{\leavevmode\hbox to3em{\hrulefill}\thinspace}
\providecommand{\MR}[1]{%
 MR \href{http://www.ams.org/mathscinet-getitem?mr=#1}{#1}.
}
\providecommand{\href}[2]{#2}
\newcommand{\articlehref}[2]{\href{#1}{#2}}

\end{document}